\documentclass[12pt]{amsart}
\usepackage{amsmath}
\usepackage{amsthm}
\usepackage{amsfonts}
\usepackage{amssymb}
\newtheorem{Theorem}{Theorem}[section]
\newtheorem{Proposition}[Theorem]{Proposition}
\newtheorem{Lemma}[Theorem]{Lemma}

\theoremstyle{definition}

\theoremstyle{remark}

\numberwithin{equation}{section}

\newcommand{\R}{{\mathbb R}}
\newcommand{\C}{{\mathbb C}}

\newcommand{\SL}{{\textrm{\rm SL}}}
\newcommand{\SO}{{\textrm{\rm SO}}}

\newcommand{\s}[2]{{\langle #1 , #2 \rangle}}
\newcommand{\tr}{{\textrm{\rm tr}\:}}

\renewcommand{\Im}{{\textrm{\rm Im}\:}}

\begin{document}

\title[Toda maps]{Polynomial Toda maps are transfer matrices}

\author{Christian Remling}

\address{Department of Mathematics\\
University of Oklahoma\\
Norman, OK 73019}
\email{christian.remling@ou.edu}
\urladdr{www.math.ou.edu/$\sim$cremling}

\date{March 1, 2025}



\begin{abstract}
We consider entire matrix functions $A(z)$ taking values in $\SL (2,\C)$. These map pairs of Herglotz functions by acting pointwise
as linear fractional transformations. The main examples of such \textit{Toda maps }are provided by transfer matrices of
differential and difference operators and by the cocycles associated with the classical integrable systems (Toda, KdV, etc.)
on these operators. Here we consider polynomial matrix functions $A(z)$. We describe these in terms of a factorization, and we then
prove that if $A$ induces a Toda map, then $A$ is essentially a transfer matrix.
\end{abstract}
\maketitle
\section{Toda maps and transfer matrices}
A \textit{Herglotz function }is a holomorphic map $F:\C^+\to\overline{\C^+}$. Here $\C^+=\{ z\in\C: \Im z>0\}$ is the upper half plane,
and the closure $\overline{\C^+}=\C^+\cup\R_{\infty}$, $\R_{\infty}=\R\cup \{\infty\}$, is taken in the Riemann sphere $\C_{\infty}$. We denote the set of Herglotz
functions by $\mathcal F$.

We are interested in matrix functions $A(z)$ from the group
\[
\mathcal{SL} = \{ A: \C\to\SL (2,\C) : A \textrm{ entire, } A(x)\in\SL (2,\R)\textrm{ for }x\in\R \} .
\]
The reason for this interest lies in the fact that such an $A\in\mathcal{SL}$ may induce a transformation
$(F_+,F_-)\mapsto (G_+,G_-)$ between pairs of Herglotz functions by acting pointwise as a linear
fractional transformation, as follows:
\begin{equation}
\label{TM}
G_{\pm}(z) =\pm \left( A(z)\cdot [\pm F_{\pm}(z)]\right) , \quad\quad z\in\C^+ .
\end{equation}
Here the dot notation refers to the natural action of $\SL(2,\C)$ on $\C_{\infty}$, which is given by
\[
\begin{pmatrix} a & b \\ c& d\end{pmatrix} \cdot w = \frac{aw+b}{cw+d} .
\]
Alternatively, we can view $\C_{\infty}\cong \C\mathbb P^1$ as projective space and thus identify vectors $v=(v_1,v_2)\in\C^2$, $v\not= 0$, with points
$z=v_1/v_2\in\C_{\infty}$, and then a $B\in\SL (2,\C)$ simply acts on $v$ as a matrix in the natural way.

Given an $A\in\mathcal{SL}$ and $F_{\pm}\in\mathcal F$, the action \eqref{TM} will always define two new holomorphic functions $G_{\pm}:\C^+\to\C_{\infty}$,
but of course there is no guarantee that $G_{\pm}$ will be Herglotz functions again.
We thus introduce the \textit{domain }of an $A\in\mathcal{SL}$ as
\[
D(A) = \{ (F_+,F_-)\in\mathcal F^2 : \pm (A\cdot (\pm F_{\pm}))\in\mathcal F \} ,
\]
and we call the correspondence $(F_+,F_-)\mapsto (G_+,G_-)$, $(F_+,F_-)\in D(A)$, a \textit{Toda map.}
This notion was introduced and advertised in \cite{Remgen,Rembook,RemToda}.
We must keep our expectations on how large $D(A)$ can be
reasonably low here since Toda maps are rather special transformations:
if we view them alternatively as maps of canonical systems, as will be discussed in a moment, then the transformed system is
unitarily equivalent to the original one, and the absolute values of the (generalized) reflection coefficients are preserved \cite[Theorems 7.2, 7.7]{Rembook}.
So the supply of $(G_+,G_-)\in\mathcal F^2$ that could conceivably be reached from a given pair $(F_+,F_-)$ by a Toda map is rather small from the outset.

We will also see below that $D(A)=\emptyset$ for many $A\in\mathcal{SL}$. At the other end of the spectrum,
the only $A\in\mathcal{SL}$ with $D(A)=\mathcal F^2$ are the constant functions $A(z)=B\in\SL (2,\R)$; then $B$ acts as an automorphism of $\C^+$.

Pairs of Herglotz functions are in one-to-one correspondence with \textit{canonical systems. }These are differential equations of the form
\begin{equation}
\label{can}
Jy'(x) = -zH(x) y(x), \quad J=\begin{pmatrix} 0 & -1 \\ 1 & 0 \end{pmatrix}, \quad x\in\R ,
\end{equation}
with Borel measurable coefficient functions $H(x)\in\R^{2\times 2}$, $H(x)\ge 0$, $\tr H(x)=1$.
It is in this context that the transformations \eqref{TM} occur naturally, in at least two ways.

First of all, the \textit{transfer matrices }$T(x,a;z)$
are in $\mathcal{SL}$. These are defined as the matrix solution of \eqref{can} with the initial value $T(a,a;z)=1$.
In this paper, it will always be understood that $x\ge a$ when discussing transfer matrices. This cases suffices since
$T(a,x)=T(x,a)^{-1}$. Also, if $a=0$, then we usually write the transfer matrix as simply $T(x,z)$.

The \textit{Titchmarsh-Weyl $m$ functions }of \eqref{can} may be defined as
\begin{equation}
\label{defm}
m_{\pm}(z) = \pm f_{\pm}(0,z) , \quad z\in\C^+ ,
\end{equation}
if we again use the convention of identifying a vector $v\not=0$ with the point $v_1/v_2\in\C_{\infty}$ on the Riemann sphere.
Here $f_{\pm}$ denotes the unique, up to a constant factor, solution of \eqref{can} that is in $L^2_H$ on $\pm x>0$. In other words,
\[
\int_0^{\infty} f^*_+(x,z) H(x) f_+(x,z)\, dx <\infty ,
\]
and similarly for $f_-$. We have $m_{\pm}\in\mathcal F$ and, conversely, for any given $F\in\mathcal F$, there is a unique coefficient
function $H(x)$ on $x>0$ such that $m_+(z;H)=F(z)$. Of course, the same fundamental result holds for $m_-$ and left half lines. See \cite[Theorem 5.1]{Rembook}.

Moreover, it is clear from the definitions of $m_{\pm}$
and the transfer matrix that if we replace $H(x)$ by its shifted version $K(x)=H(x+L)$, then the transformation from the original $m$ functions
$m_{\pm}=m_{\pm}(z;H)$ to the new ones $M_{\pm}=m_{\pm}(z;K)$ is obtained by letting $T=T(L,z)$ act as in \eqref{TM}: $M_{\pm}=\pm (T\cdot (\pm m_{\pm}))$.

More sophisticated examples of Toda maps (which also explain the terminology) are obtained from the classical integrable systems on difference and differential
operators such as the flows from the Toda hierarchy (on Jacobi matrices) or the KdV hierarchy (on Schr{\"o}dinger operators). The subject is discussed
in \cite{OngRem,RemToda} in some detail from precisely this point of view. Please also consult \cite{GesHol,Kot,Teschl} for more background information.

It is this connection that provided the original motivation for the present work. There is some
evidence \cite{HurOng,RemToda} that the usual constructions of these hierarchies run into considerable obstacles in the more general framework of canonical systems.
The question of whether (and how) such hierarchies could be constructed seems quite fundamental but, to my knowledge, has received little attention so far beyond the
attempts in \cite{HurOng,RemToda}.

If one subscribes to the point of view advertised in \cite{OngRem,RemToda},
then the key feature of these systems is the associated cocycle of matrix functions $A(z)\in\mathcal{SL}$
that may be used to implement the dynamics by letting them act as Toda maps. It is now natural to adopt a more abstract approach and inquire about Toda maps in general.
This paper presents the attempt to get this program started by looking at what must be the most basic case, namely that of matrix functions $A(z)$ with polynomial
dependence on $z$.

We will obtain rather complete answers in this case, and perhaps this can be the first step of a larger project.
We will prove that there are no new examples of such polynomial Toda maps beyond the obvious ones of
transfer matrices with polynomial dependence on $z$ (see Theorem \ref{T1.2} below), even though there is a large supply of polynomial matrix functions $A(z)$
that are not transfer matrices (see Theorem \ref{T1.1} for this).
However, these will turn out to have empty domains $D(A)$.

In general, an $A(z)\in\mathcal{SL}$ is a transfer matrix if and only if it satisfies the additional conditions $A(0)=1$ and
\begin{equation}
\label{hp}
i( A^*(z)JA(z)-J)\ge 0 \textrm{ for all }z\in\C^+ .
\end{equation}
This second condition \eqref{hp} is equivalent to $w\mapsto A^{-1}(z)\cdot w$ being a Herglotz function for all (fixed) $z\in\C^+$
\cite[Lemma 3.9]{Rembook}. It is fairly straightforward
to verify that a transfer matrix $A(z)=T(L,z)$ satisfies these extra conditions; the converse is another major result from the inverse spectral theory of
canonical systems. Please see \cite[Section 4.4, Theorem 5.2]{Rembook} for further discussion.

Condition \eqref{hp} also is the key ingredient to the results on reflectionless limit points \cite{RemAnn}; in the more general context of integrable flows and Toda maps,
these issues have recently been studied in depth by Kotani \cite{Kot}.

We introduce the notation $\mathcal{TM}$ for this subclass of $\mathcal{SL}$, so we define
\[
\mathcal{TM} = \{ A\in\mathcal{SL}: A(0)=1, A \textrm{ satisfies }\eqref{hp} \} .
\]
As announced, we restrict our attention here to polynomial matrix functions
\[
\mathcal P = \{ A\in\mathcal{SL}: A(z) = 1+zA_1+ \ldots + z^n A_n \} .
\]
I have kept the normalization $A(0)=1$, but this is only for convenience and not essential for what follows. A general polynomial $A(z)\in\mathcal{SL}$
can be written as $A(z)=A(0)(A(0)^{-1}A(z))=(A(z)A(0)^{-1})A(0)$, so differs from an element of $\mathcal P$ only by the constant matrix
$A(0)\in\SL (2,\R)$. This acts as an automorphism of $\C^+$, so in particular will not affect
the basic question of whether the domain $D(A)$ of the associated Toda map is non-empty.

We can describe the class of polynomial transfer matrices $T\in\mathcal P\cap\mathcal{TM}$ very explicitly. The general such $T$ is given by
\begin{equation}
\label{1.1}
T(z) = (1+L_1zJP_1)\cdots (1+L_N zJP_N) ,
\end{equation}
and here $L_j>0$, and each $P_j=P_{\alpha_j}$ is a projection, onto some $e_{\alpha_j}$, with $e_{\alpha}=(\cos\alpha,\sin\alpha)^t$;
compare \cite[Lemma 5.9]{Rembook}. So
\[
P_{\alpha} = e_{\alpha}e^*_{\alpha} = \begin{pmatrix} \cos^2\alpha & \sin\alpha\cos\alpha \\ \sin\alpha\cos\alpha & \sin^2\alpha \end{pmatrix} .
\]
In the sequel, by a \textit{projection }$P$ we will always mean a matrix of this form.

We can be more explicit still. The matrix function $S(z)=1+LzJP$ is the transfer matrix $S(z)=T(L,z)$ of the coefficient function $H(x)=P$ on $0\le x\le L$;
such an interval on which $H$ is a constant projection is called a \textit{singular interval. }See also \cite[Section 1.2]{Rembook}.
So the polynomial transfer matrices are obtained by solving \eqref{can} across a finite number of singular intervals.

We will prove the following version of \eqref{1.1} for general $A\in\mathcal P$.
\begin{Theorem}
\label{T1.1}
Let $A\in\mathcal P$. Then there are non-zero polynomials $p_j(z)$, $p_j(0)=0$, with real coefficients and projections $P_j$ satisfying $P_j\not= P_{j+1}$ such that
\begin{equation}
\label{1.7}
A(z) = (1+p_1(z)JP_1) \cdots (1+p_N(z)JP_N) .
\end{equation}
This factorization is unique. Conversely, any such product defines an $A\in\mathcal P$.
\end{Theorem}
This clarifies the relation between $\mathcal P$ and the smaller subclass $\mathcal P\cap\mathcal{TM}$. The factorization \eqref{1.7} is more general
than \eqref{1.1} in two ways: the polynomials $p_j$ can be of degree larger than one, and even if $\deg p_j=1$, so $p_j(z)=Lz$,
then we can have $L<0$ so that then the corresponding factor $1+LzJP$ is not a transfer matrix across a singular interval but rather the inverse of such a matrix.
So already in this special case $\deg p_j=1$, we are dealing with a considerably larger class of matrix functions, which can now be built from the basic
transfer matrices $S(z)=1+LzJP$ \textit{and their inverses }in arbitrary succession.

For transfer matrices $A(z)=T(L,z)$, there is an easy complete description of $D(A)$. We have $(F_+,F_-)\in D(A)$ if and only if
$F_+(z)$ lies in the \textit{Weyl disk }$A^{-1}(z)\overline{\C^+}$ for all $z\in\C^+$, and $F_-$ can be an arbitrary Herglotz function. Equivalently,
$F_+$ must be the half line $m$ function of a canonical system $K(x)$, $x\ge 0$, that agrees with $H(x)$ on $0\le x\le L$ if $A(z)=T(L,z;H)$ was the transfer
matrix of $H$ across $[0,L]$. Please see \cite[Theorem 6.1]{Rembook} for a discussion of these facts.

In particular, we see that always $D(T)\not=\emptyset$ for a transfer matrix $A(z)=T(z)$ (whether a polynomial of $z$ or not), and in fact these
domains are rather large.

Observe also that it is clear that there is no condition on $F_-$. We can implement multiplication by $-1$ by letting the matrix
\[
I = \begin{pmatrix} 1 & 0 \\ 0 & -1 \end{pmatrix}
\]
act, so $-(A\cdot (-F))=(IAI)\cdot F$, and now \eqref{hp} also implies that $w\mapsto IA(z)I\cdot w$ is a Herglotz function for $z\in\C^+$.
Compare \cite[Lemma 4.14]{Rembook}.

The main result of this paper was stated in its title. It essentially says that if $A\in\mathcal P$ and $D(A)\not=\emptyset$, then $A$ is a transfer matrix or the inverse
of a transfer matrix, so the $T\in\mathcal P\cap\mathcal{TM}$ already give us all the polynomial Toda maps.
However, the statement is not true in literally this form due to the presence of the trivial example
\begin{equation}
\label{1.2}
A(z)= 1+p(z)JP .
\end{equation}
For any polynomial $p$, if $F_{\pm}(z)\equiv \pm x\in\R_{\infty}$, with $x\in\R_{\infty}$ denoting the extended real number represented by the
unit vector $v\in N(P)$, then trivially $A\cdot F_+=F_+$, $IAI\cdot F_-=F_-$.
So this is the identity transformation, and we will see later, in Lemma \ref{L3.2} below, that $D(A)=\{ (x,-x)\}$ if $\deg p\ge 2$, so this $A$ can not be applied
to anything else and is thus
completely uninteresting as a Toda map. Moreover, this pair of constant real $m$ functions $(x,-x)$ corresponds to the canonical system with coefficient function
$H(x)\equiv P$, which is also trivial from a spectral theoretic point of view. On the other hand, if $\deg p=1$ in \eqref{1.2}, then $A$ or $A^{-1}$ is a transfer
matrix across a single singular interval.
\begin{Theorem}
\label{T1.2}
Let $A\in\mathcal P$, and suppose that $D(A)\not=\emptyset$ and $A$ is not of the type \eqref{1.2} with $\deg p\ge 2$.
Then $A\in\mathcal{TM}$ or $A^{-1}\in\mathcal{TM}$.
\end{Theorem}

So, to summarize the whole plot, while there are many polynomial matrix functions $A\in\mathcal P$ that are not transfer matrices, these can not be applied
to anything, so do not induce Toda maps and thus the transfer matrices suffice if we are interested in the maps. I should perhaps also mention one more time that the
situation is completely different once the assumption of polynomial dependence on $z$ is dropped because then the classical flows do provide examples of Toda maps
that are not induced by transfer matrices.

Sections 2 and 3 will present the proofs of Theorems \ref{T1.1} and \ref{T1.2}, respectively.
\section{Proof of Theorem \ref{T1.1}}
The matrix $J$ acts as a rotation by 90 degrees, so $PJP=0$ for any projection $P$. As a consequence, $e^{pJP}=1+pJP$, and since $\tr JP=0$, this shows that
$\det (1+pJP)=1$. It is now clear that the product from Theorem \ref{T1.1} defines a matrix function from $\mathcal P$.

Conversely, suppose that an $A\in\mathcal P$ is given, and to avoid trivialities, assume also that $n=\deg A\ge 1$. If we expand $A(z)=1+\ldots +z^n A_n$,
then $\det A_n=0$. We bring $A_n\in\R^{2\times 2}$ to Jordan normal form $SA_nS^{-1}$, which we can do by changing bases by
an $S\in\SL( 2,\R)$ since $A_n$ has the real eigenvalues $0$, $\tr A_n$.
Let's first deal with the case when $A_n$ is diagonalizable. Consider the transformed $B\in\mathcal P$, $B(z)=SA(z)S^{-1}$.
We can assume that
\begin{equation}
\label{2.1}
B(z) = \begin{pmatrix} az^n & bz^s \\ cz^t & dz^k \end{pmatrix} + \textrm{ lower order terms} .
\end{equation}
More precisely, what we mean by this is that the first matrix displays the highest order term in each entry separately; it is of course quite likely
that the next term in for example the $(1,1)$ element is $ez^{n-1}$, which is not of lower order than $dz^k$, say.
We have $a,b,c,d\not= 0$, $s,t,k\le n-1$, and also $n+k=s+t$ and $ad=bc$, or the highest order term in the expansion of $\det B(z)$ would not equal zero.
Hence the degree of
\[
C(z)=\begin{pmatrix} 1 &  -\frac{a}{c}z^{n-t} \\ 0 & 1 \end{pmatrix} B(z)
\]
is at most $n-1$, and we have successfully factored
\begin{equation}
\label{2.6}
B(z) = \begin{pmatrix} 1 &  \frac{a}{c}z^{n-t} \\ 0 & 1 \end{pmatrix} C(z) .
\end{equation}
Observe now that the first matrix on the right-hand side of \eqref{2.6}
equals $1-\frac{a}{c}z^{n-t}JP_2$, with $P_2=\bigl( \begin{smallmatrix} 0 & 0\\0&1\end{smallmatrix}\bigr)$
denoting the projection onto $e_2$. We then return to $A(z)=S^{-1}B(z)S$. Recall that any $S\in\SL(2,\R)$ satisfies the identity $S^{-1}J=JS^t$. Thus
\[
S^{-1}(1+p(z)JP_2)S=1+p(z)JS^tP_2S= 1+rp(z)JQ ,
\]
and here $Q$ is another projection, onto $S^te_2$, and $r>0$. This simple fact will be used frequently in the sequel, so let me state it separately.
\begin{Lemma}
\label{L2.2}
Let $S\in\SL (2,\R)$ and let $P$ be the projection onto $v\in\R^2$. Then $S^{-1}JPS =cJQ$, with $c>0$, and $Q$ is the projection onto $S^tv$.
\end{Lemma}

We conclude that we have split off a factor of the desired form in the original matrix function also: we have
\[
A(z) = (1+kz^m JQ) D(z) ,
\]
with $D\in\mathcal P$, $\deg D\le n-1$.

The same procedure works if $A_n$ is not diagonalizable. Since $\det A_n=0$, this matrix now has $\lambda=0$ as its only eigenvalue.
We thus obtain the following analog of \eqref{2.1}:
\[
B(z) = \begin{pmatrix} az^s & \pm z^n \\ cz^k & dz^t \end{pmatrix} + \textrm{ lower order terms} .
\]
We can then lower the degree by multiplying from the left by
\[
1  \pm \frac{1}{d}z^{n-t} JP_2 = \begin{pmatrix} 1 & \mp\frac{1}{d} z^{n-t} \\ 0& 1 \end{pmatrix}
\]
and otherwise argue as above.

Finally, we of course repeat this whole basic step with the new matrix function $D(z)$ etc.\ until the degree is zero.
It could happen here that successive factors contain the same projection but this is not a problem because $(1+pJP)(1+qJP)=1+(p+q)JP$, so if it does happen,
we simply combine these factors into a single one.

It remains to prove the uniqueness of these factorizations.
\begin{Lemma}
\label{L2.1}
Given projections $P_1,\ldots , P_N$, we have
\[
JP_1JP_2 \cdots JP_N\not=0
\]
if and only if $P_j\not= P_{j+1}$ for $j=1,2, \ldots, N-1$.
\end{Lemma}
\begin{proof}
One direction is obvious because, as we observed earlier, $PJP=0$ for any projection $P$.

Conversely, suppose now that $P_j\not= P_{j+1}$.
We can argue by induction on $N$, so we can assume that $JP_1\cdots JP_{N-1}\not =0$. Since this matrix is singular, its kernel must
be one-dimensional and is thus equal to $N(P_{N-1})$. Since $J$ is rotation by 90 degrees, we have
$R(JP_N)=R(P_N)^{\perp} = N(P_N)$. By assumption, this is not the same space as $N(P_{N-1})$ and thus $(JP_1\cdots JP_{N-1})JP_N\not=0$, as claimed.
\end{proof}
Lemma \ref{L2.1} shows that if an $A\in\mathcal P$ with $\deg A=n$ is factored as stated in Theorem \ref{T1.1}, then the coefficient (matrix) $A_n$ of $z^n$ is a multiple
of $JP_1\cdots JP_N$. In particular, $P_N$ can be recovered from $A(z)$ via $A_n$ as the projection with the same null space as $A_n$. So a second factorization
\[
A(z) = (1+q_1JQ_1)\cdots (1+q_M JQ_M)
\]
would have to satisfy $Q_M=P_N$. Thus if we multiply by the inverse $1-p_NJP_N$ of the last factor of the first factorization, we obtain
\[
(1+p_1JP_1) \cdots (1+p_{N-1}JP_{N-1}) = (1+q_1JQ_1) \cdots (1+(q_M-p_N)JP_N) .
\]
If we had $q_M\not= p_N$, then the same argument would now show that $P_{N-1}=P_N$, which contradicts our assumptions. So we also have $q_M=p_N$.
In other words, the rightmost factors agree, and then we can of course continue in this style to deduce that the whole factorizations must be identical.
This concludes the proof of Theorem \ref{T1.1}.

A key step of this proof (and also the one of Theorem \ref{T1.2}, to be discussed in the next section) was to bring
$A_n$ or, equivalently, $JP_1\cdots JP_N$, to Jordan normal form. Therefore the following observation is perhaps of some interest even though we will
not need it here.
\begin{Proposition}
\label{P2.1}
Let $P_1, \ldots , P_N$, $N\ge 1$, be projections and suppose that $P_j\not= P_{j+1}$, $j=1,2, \ldots , N-1$. Then $JP_1 \cdots JP_N$ is diagonalizable if and only
if $P_1\not= P_N$.
\end{Proposition}
\begin{proof}
Call this matrix $B$. Recall that $\det B=0$, $B\not=0$, so the eigenvalues of $B$ are $0$, $\tr B$, and thus $B$ is diagonalizable if and only if $\tr B\not= 0$.
We compute the trace using an orthonormal basis $\{ e, f\}$ satisfying $P_Ne=0$, $f=Je$. This gives
\[
\tr B = \s{f}{Bf} =\s{Je}{Bf} = -\s{e}{JBf} .
\]
As in the proof of Lemma \ref{L2.1}, we have $JBf\not= 0$. Since $J^2=-1$, we see that $JBf\in R(P_1)=N(P_1)^{\perp}$. This is orthogonal to $e\in N(P_N)$
if and only if $P_1=P_N$.
\end{proof}
\section{Proof of Theorem \ref{T1.2}}
Let's start by reviewing the one basic fact about Herglotz functions that we will need here. Namely, if $F\in\mathcal F$, $F\not\equiv\infty$,
then $b=\lim_{y\to\infty} -iF(iy)/y$ exists and $b\ge 0$. This is well known and also an immediate consequence of the Herglotz representation formula
\[
F(z) = a + \int_{\R_{\infty}} \frac{1+tz}{t-z}\, d\nu(t) .
\]
This property may also be applied to $-1/F$, which is another Herglotz function, so it is also true that if $F\in\mathcal F$, $F\not\equiv 0$,
then either $y|F(iy)|\to\infty$ or else $c=\lim_{y\to\infty} -iy\,F(iy)$ exists and $c>0$.

Or, to summarize this somewhat imprecisely but more intuitively, Herglotz functions cannot grow faster than $bz$ and they cannot decay more rapidly than $-c/z$
for large $z$. This will become important when we analyze later how they could alter the asymptotics of the polynomial entries of an $A\in\mathcal P$.

Let's then start the proof of Theorem \ref{T1.2} with some observations on more specialized situations.
\begin{Lemma}
\label{L3.1}
Suppose that $A(z)=T_1(z)T^{-1}_2(z)$, with $T_j\in\mathcal{TM}$, and $A\cdot F = G$ for some $F,G\in\mathcal F$.
Then $A\in\mathcal{TM}$ or $A^{-1}\in\mathcal{TM}$.
\end{Lemma}
\begin{proof}
We have $T^{-1}_2\cdot F = T^{-1}_1\cdot G$. As we discussed earlier, the Herglotz functions $\{ T(L,z;H)^{-1}\cdot M: M\in\mathcal F\}$ are exactly the
half line $m$ functions of those canonical systems whose coefficient function agrees with $H(x)$ on $0\le x\le L$; see \cite[Theorem 6.1]{Rembook} again.
Thus in our situation, if we write $T_j(z)=T(L_j,z; H_j)$ and $L_2\ge L_1$, say, then $H_1(x)=H_2(x)$ on $0\le x\le L_1$. This means that
$T_2=T_3T_1$, and here $T_3(z)=T(L_2,L_1,z; H_2)$ also lies in $\mathcal{TM}$. In the other case, when $L_1>L_2$, we similarly obtain
$T_1=T_4T_2$. In either case, there is a cancellation in the product defining $A(z)$, and the statement follows.
\end{proof}
\begin{Lemma}
\label{L3.2}
(a) Let $A\in\mathcal P$, and suppose that $A\cdot x\in\mathcal F$ for some $x\in\R_{\infty}$. Then $A$ is of the form $A(z)=T^{-1}(z)(1+p(z)JP)$,
with $T\in\mathcal{TM}\cap\mathcal P$ and $Pv=0$ for the vector $v$ representing $x$.

(b) If, in addition, $D(A)\not=\emptyset$, then $A\in\mathcal{TM}$ or $A^{-1}\in\mathcal{TM}$, or else $A=1+pJP$ with $\deg p\ge 2$ and $D(A)=\{ (x,-x) \}$.
\end{Lemma}
We can take the vector $v$ from part (a) as $v=(x,1)^t$ if $x\not=\infty$ and $v=(1,0)^t$ if $x=\infty$.
\begin{proof}
(a) The function $A(z)\cdot x$ is rational, and the rational Herglotz functions are exactly those of the form $T^{-1}(z)\cdot y$, with
$T\in\mathcal{TM}\cap\mathcal P$ and $y\in\R_{\infty}$. See \cite[Lemma 5.9]{Rembook}. So if $v,w\in\R^2$ denote vectors representing $x$ and $y$,
respectively, then $A(z)v=\lambda(z)T^{-1}(z)w$ or $TAv=\lambda w$. Since $A(0)=T(0)=1$, this is only possible if $v=w$.
So we now have a matrix function $B=TA\in\mathcal P$ satisfying $B(z)v=\lambda(z) v$. If we work with $C(z)=RB(z)R^t$ instead, for a rotation $R\in\SO (2)$ that maps
$Rv=e_1$, then we will have $C(z)e_1=\lambda(z) e_1$. In other words $C_{21}(z)=0$. Clearly, the only such $C\in\mathcal P$ are
\[
C(z)= \begin{pmatrix} 1 & -p(z) \\ 0 & 1 \end{pmatrix} = 1 + p(z)JP_2 .
\]
We can now transform back. Since $J$ is a rotation itself and thus commutes with $R$, we have
$R^tJP_2R= JP$, and here $P=R^tP_2R$ still is a projection, and $Pv=PR^te_1=0$. This confirms that $A$ is of the asserted form.

(b) It is easy to verify, using \eqref{hp} or the reformulation mentioned above, that if $T\in\mathcal{TM}$, then $T_1=IT^{-1}I\in\mathcal{TM}$ also.
See \cite[Lemma 4.14]{Rembook} again.

Thus part (a) shows that $IAI$ is of the form $IAI=T_1(1-pJQ)$, with $T_1\in\mathcal{TM}\cap\mathcal P$.
If $\deg p\le 1$ here, then this matrix function is of the form $IAI=T_1T_2^{\pm 1}$. However, Lemma \ref{L3.1} now shows that the minus sign
is only possible if there is a cancellation or one of the two matrix functions is the identity matrix.
In every case, it turns out that $IAI$ or its inverse lies in $\mathcal{TM}$ and thus the same is true of $A$ itself.

It remains to discuss the case $\deg p\ge 2$.
By assumption, there are $F,G\in\mathcal F$ such that $IAI\cdot F=G$. We can write this in the form
$K = (1-pJQ)\cdot F$, and here $K=T^{-1}_1\cdot G$ is another Herglotz function. We again rotate such that $R^tQR=P_2$. Then $R^t\cdot K =
(1-pJP_2)\cdot R^t F$, and if $M=R^t\cdot F\not\equiv\infty$, then this equals $M(z)+p(z)$. Since $\deg p\ge 2$ now, this contradicts the asymptotic
behavior of Herglotz functions that was discussed at the beginning of this section. We conclude that $F\equiv y=R\cdot\infty$, and here $y\in\R_{\infty}$
can also be characterized as the unique number for which the corresponding vector $v\in\R^2$ satisfies $Qv=0$. We have also shown that $F\equiv y$
is the only Herglotz function to which $IAI$ can be applied.

Moreover, $T_1^{-1}\cdot G=y$ also, and now we again make use of the description of the Herglotz functions of this type from \cite[Theorem 6.1]{Rembook}.
Since $y$ is the $m$ function of the constant coefficient function $H(x)=Q$, we see that $T_1=1+LzJQ$ is the transfer matrix across an interval of this coefficient
function. This shows that
\[
IAI=T_1(1-pJQ)=1+(Lz-p)JQ=1+qJQ
\]
and thus $A=1-qJP$ as well. We find ourselves in the last case from the statement of Lemma \ref{L3.2}(b), and we can finally analyze $A$ itself in the same way
to confirm that $D(A)$ is as described.
\end{proof}

We now come to the main step. Let us introduce two more pieces of notation:
\begin{gather*}
\mathcal F_1 =\mathcal F\setminus \{ F\equiv x: x\in\R_{\infty}\} , \\
\mathcal T_1 = \{ 1+LzJP : L>0, P \textrm{ projection}\} .
\end{gather*}
So $\mathcal F_1$ is the collection of genuine Herglotz functions $F:\C^+\to\C^+$, with those functions that are identically equal to an extended
real number $x\in\R_{\infty}$ excluded, and $\mathcal T_1\subseteq\mathcal{TM}\cap\mathcal P$ is the collection of degree one transfer matrices,
across a single singular interval.
\begin{Lemma}
\label{L3.3}
Let $A\in\mathcal P$, with $n=\deg A\ge 2$. Suppose that $(F_+,F_-)\in D(A)$, with $F_{\pm}\in\mathcal F_1$ and $G_+=A\cdot F_+, G_-=IAI\cdot F_-\in\mathcal F_1$ also.

Then either $A=T_1BT_2$, with $T_j\in\mathcal T_1$, $B\in\mathcal P$, $\deg B=n-2$, and $F_+=T_2^{-1}\cdot K_+$ for some $K_+\in\mathcal F$,
or else $IAI$ is of this form and the corresponding statements hold for this matrix function and $F_-$.
\end{Lemma}
Before we prove this, let's discuss how Lemma \ref{L3.3} can be used to establish Theorem \ref{T1.2}. Let $A\in\mathcal P$, with $D(A)\not=\emptyset$, be given.
We then have four Herglotz functions $F_{\pm}$, $G_{\pm}$ such that $A\cdot F_+=G_+$, $IAI\cdot F_-=G_-$. If at least one of the four functions is not in
$\mathcal F_1$ here, then Lemma \ref{L3.2}, applied to $A$ or to $A^{-1}$, will give the desired conclusions.

So we can focus on the case $F_{\pm},G_{\pm}\in\mathcal F_1$, and now Lemma \ref{L3.3} applies. Let's say we are in the first case, so $A=T_1BT_2$.
Since $F_+=T_2^{-1}\cdot K_+$, we have $T_1B\cdot K_+\in\mathcal F_1$. Moreover, $I(T_1B)I\cdot (IT_2I\cdot F_-)\in\mathcal F_1$ as well,
and here $IT_2I=T_3^{-1}$ is the inverse of a $T_3\in\mathcal T_1$, so in particular $IT_2I\cdot F_-=T^{-1}_3\cdot F_-\in\mathcal F$.

The upshot of all this is that we may remove the last factor $T_2$ of the factorization of $A$, and the reduced matrix still has a non-empty domain
because, as we just saw, $(K_+,K_-)\in D(T_1B)$, with $K_+=T_2\cdot F_+$, $K_-=T^{-1}_3\cdot F_-$. The images $(G_+,G_-)$ haven't changed.

It could happen here that $K_+=x\in\R_{\infty}$, so $K_+\notin\mathcal F_1$, but this only opens up a short-cut to our eventual goal.
(We do always have $K_-\in\mathcal F_1$, as we can confirm by comparing with Lemma \ref{L3.2} or by using the description of the Herglotz functions
$\{ T^{-1}_3\cdot N: N\in\mathcal F\}$ again.)
We may now refer to Lemma \ref{L3.2}(b), and only the case $T_1B\in\mathcal{TM}$ is consistent with what we already know about the factorization
of this matrix function. In particular, observe that if we had $T_1B=1+pJP$ with $\deg p\ge 2$, then factoring out a $T_1\in\mathcal T_1$
would not reduce the degree, so this is impossible here. We conclude that $A=T_1BT_2\in\mathcal{TM}$, as claimed.

In the other case, when $K_{\pm}\in\mathcal F_1$, we can apply Lemma \ref{L3.3} again,
this time to $T_1B=T_1CT_4$, to remove one more factor $T_4\in\mathcal T_1$ and reducing the
degree at the same time.

We continue in this style until we have completely factored $A$ as a product of matrices from $\mathcal T_1$. In particular, $A\in\mathcal{TM}$, as desired.
It is of course not possible here that we suddenly find ourselves in the other case of Lemma \ref{L3.3} in the middle of this process.
More explicitly, we cannot have, say, $I(T_1B)I=T_5DT_6$ rather than $T_1B=T_1CT_4$ because $IT_5I\notin\mathcal T_1$ if $T_5\in\mathcal T_1$, so this
would contradict the uniqueness of such factorizations.

The other case, when $IAI=T_1BT_2$ initially, is completely analogous. This time, the method will show that $A^{-1}\in\mathcal{TM}$.
\begin{proof}[Proof of Lemma \ref{L3.3}]
We again bring the highest order coefficient $A_n$ of $A(z)=1+\ldots + z^n A_n$ to Jordan normal form, so work with $SA(z)S^{-1}$, for suitable
$S\in\SL (2,\R)$. Let's first assume that $A_n$ is diagonalizable. Then again, as in \eqref{2.1},
\begin{equation}
\label{3.2}
SA(z)S^{-1} = \begin{pmatrix} az^n & bz^s \\ cz^t & dz^k \end{pmatrix} + \textrm{ lower order terms} .
\end{equation}
Here $a,b,c,d\not =0$, $ad=bc$, $n+k=s+t$, $s,t\le n-1$, $k\le n-2$. We can extract additional information on the degrees from
the extra assumptions on the existence of non-trivial Herglotz functions in the domain. I claim that $s=n-1$. If, on the contrary, we had $s\le n-2$, then
\begin{equation}
\label{3.1}
SA(z)S^{-1}\cdot (SF_+) = \frac{ az^n K(z) + O(z^{n-1}K(z))+O(z^{n-2})}{cz^t K(z) + O(z^{t-1}K(z))+O(z^{t-2})} ,
\end{equation}
and here we have written $K=S\cdot F_+$.
This follows because, by assumption, $s\le n-2$, and thus also $k=t+s-n\le t-2$.
Note also that since $S\in\SL (2,\R)$ acts as an automorphism of $\C^+$, we still have $K\in\mathcal F_1$. Furthermore,
$SA\cdot F_+\in\mathcal F_1$, for the same reason.

We now make use of the fact, reviewed at the beginning of this section, that the Herglotz function $K\not\equiv 0$ can not decay faster than $c/z$. Hence
the right-hand side of \eqref{3.1} has the asymptotic behavior $\simeq (a/c)z^{n-t}$. This could be compatible with this function being a Herglotz function
if $t=n-1$, $a/c>0$. However, we can then run the same analysis for $IAI\cdot F_-$, and this time we obtain the asymptotics $(a/(-c))z$
because $ISAS^{-1}I=(ISI) IAI (ISI)^{-1}$ is the same matrix as $SAS^{-1}$, but with the signs changed in the off-diagonal elements.
We have run into a contradiction after all. We must admit that $s=n-1$. Finally, a similar argument may be applied to $A^{-1}$, and this will show that $t=n-1$ also.

So we now have the following more precise version of \eqref{3.2}:
\begin{equation}
\label{3.3}
SA(z)S^{-1} = \begin{pmatrix} az^n & bz^{n-1} \\ cz^{n-1} & dz^{n-2} \end{pmatrix} + \textrm{ lower order terms} .
\end{equation}
Moreover, as above, by making the appropriate
choice between $A$ and $IAI$, we may also assume that $a/c<0$. Let's assume, for convenience, that this happens for $A$ itself.

We can now repeat the asymptotic analysis of $SAS^{-1}\cdot K$ from above. The contradictory asymptotics $(a/c)z+o(z)$ can only be avoided if the terms
$az^nK(z)$ and $cz^{n-1}K(z)$ do not dominate all other contributions from the numerator and denominator, respectively.
This in turn is only possible if $K(z)=-\gamma/z+o(1/z)$, with $\gamma>0$. In this situation, \eqref{3.1} becomes
\[
SA(z)S^{-1}\cdot (SF_+) = \frac{ (-\gamma a+b)z^{n-1} +o(z^{n-1})}{(-\gamma c+d)z^{n-2} +o(z^{n-2})} .
\]
Since $ad=bc$, we see from this that we must insist that $\gamma=b/a=d/c$, or else we would still obtain the asymptotics $(a/c)z$. In particular, $b/a>0$.

Recall that if $K=m_+(z;H)$ is viewed as the half line $m$ function of a canonical system, then such asymptotics $K(z)\simeq -\gamma/z$ are equivalent
to $H(x)$ starting with a singular interval of the type $H(x)=P_1$. Put differently, we have
$K(z)=(1-(z/\gamma)JP_1)\cdot M(z)$, for some $M\in\mathcal F$. See \cite[Theorem 4.33]{Rembook}
and also the argument from the proof of Theorem 4.34 there for further details on this step.

We now compare this information with what we know about the factorization of $SAS^{-1}$ from the proof of Theorem \ref{T1.1}. Recall that \eqref{3.3}
lets us reduce the degree by pulling out a factor on the left as follows:
\[
SA(z)S^{-1} = \begin{pmatrix} 1 & \frac{a}{c}z \\ 0 & 1 \end{pmatrix} C(z) .
\]
Alternatively, we could have factored out on the right, and since both factors reduce the degree, the uniqueness of such factorizations implies that we can do \textit{both.}
So we have
\[
SA(z)S^{-1} = \begin{pmatrix} 1 & \frac{a}{c}z \\ 0 & 1 \end{pmatrix} D(z) \begin{pmatrix} 1 & 0 \\ \frac{a}{b}z & 1 \end{pmatrix} ,
\]
with $\deg D=n-2$. Notice that the other two factors are equal to $1-(a/c)zJP_2$ and $1+(a/b)zJP_1$, respectively,
and since $a/c<0$, $a/b>0$, they both lie in $\mathcal T_1$.
We just saw that $K=(1+(a/b)zJP_1)^{-1}\cdot M$, so we have established the claims of Lemma \ref{L3.3} for $SAS^{-1}$. We then obtain them for
$A$ itself by transforming back and referring to Lemma \ref{L2.2} to make sure that the basic properties of the individual factors are preserved.

Finally, the treatment of the other case, when $A_n$ is not diagonalizable, is completely analogous. As in the proof of Theorem \ref{T1.1},
I again leave the details of this case to the reader.
\end{proof}

\end{document}